\documentclass{amsart}
\usepackage[all]{xy}
\SelectTips{cm}{}
\usepackage{amsmath}
\usepackage{amssymb}
\usepackage{amscd}
\usepackage{amsthm}
\usepackage{amsfonts}
\usepackage{amsxtra}

\theoremstyle{plain}

\newtheorem{Thm}{Theorem}[section]

\newtheorem{Pro}[Thm]{Proposition}

\newtheorem{bigthm}{Theorem}

\theoremstyle{definition}
\newtheorem{Def}[Thm]{Definition}

\newtheorem{Rem}[Thm]{Remark}

\newtheorem*{Rem-intro}{Remark}

\newcommand{\Dirlim}{\varinjlim}

\newcommand{\HH}{{\mathcal{H}}}
\newcommand{\UU}{{\mathcal{U}}}
\newcommand{\PU}{{\mathcal{PU}}}

\newcommand{\ZZ}{{\mathbb{Z}}}
\newcommand{\QQ}{{\mathbb{Q}}} 

\newcommand{\CC}{{\mathbb{C}}}

\newcommand{\RR}{{\mathbb{R}}}
 \newcommand{\PP}{{\mathbb{P}}}

\newcommand{\GL}{{\mathrm{GL\,}}}
\newcommand{\GLo}{{\mathrm{GL_o\,}}}
\newcommand{\U   }{{\mathrm{U}}}
\newcommand{ \Uo  }{{\mathrm{U_o\,}}}

\newcommand{\KK}{{\mathcal{K}}}

 \newcommand{\FF}{{\mathcal{F}}}
\newcommand{\esp}{{enhanced Samelson product}}

\def\:{\colon\!}
\catcode`@=11
\@xp\let\csname subjclassname@1991\endcsname \subjclassname
\@namedef{subjclassname@2010}{%
  \textup{2010} Mathematics Subject Classification}
\catcode`@=12

\begin{document}

\title[  Banach algebras, Samelson products,  and
 the Wang Differential]
{   Banach algebras, Samelson products,  and
 the Wang Differential }


\author[Schochet]{Claude L.~Schochet}
\address{Department of Mathematics,
     Wayne State University,
     Detroit MI 48202}
\email{claude@math.wayne.edu}

\thanks{ }
\begin{abstract}
Supppose given a principal $G$ bundle $\zeta : P \to S^k$ (with $k \geq 2$) and a Banach algebra $B$ upon which $G$ acts
continuously.
Let
\[
\zeta\otimes B  : \qquad  P \times _G B \longrightarrow S^k
\]
denote the associated bundle and let
\[
A_{\zeta\otimes B}  = \Gamma ( S^k, P \times _G B  )
\]
 denote the associated Banach algebra of sections. Then $\pi _*\GL A_{\zeta \otimes B}  $
is determined by a mostly degenerate spectral sequence and by a Wang differential
\[
d_k : \pi _*(\GL B )  \longrightarrow \pi _{*+k-1} (\GL B) .
\]
We show that if $B$ is a $C^*$-algebra then the differential is given explicitly in terms of an \esp\, with the clutching map of the principal bundle. Analogous
results hold after localization  and in the setting of topological $K$-theory. 

We illustrate our technique with a close analysis of the invariants associated to 
the $C^*$-algebra of sections of the bundle
\[
\zeta\otimes M_2 : \qquad S^7 \times _{S^3} M_2  \to S^4
\]
 constructed from the Hopf bundle $\zeta: \,S^7 \to S^4$ and by the  conjugation action of $S^3$ on $M_2 = M_2(\CC )$.  We compare and contrast
 the information obtained from the homotopy groups $\pi _*(A_{\zeta\otimes M_2})$, 
 the rational homotopy groups $\pi _*(A_{\zeta\otimes M_2})\otimes\QQ $ and the topological $K$-theory groups $K_*(A_{\zeta\otimes M_2})$.

\end{abstract}
\maketitle
\tableofcontents
\section{Introduction}

The group $\GL A$ of invertible elements of a Banach algebra $A$ has the homotopy type of a CW-complex, and hence its homotopy groups are in principle computable. We know
that these groups hold a lot of information about $A$, since the topological $K$-theory groups $K_*(A)$ are given by the stabilized groups $\pi _*(\GL(A\otimes \KK))$. The
groups $\pi _*(\GL A)$ are far richer in information but also far more difficult to compute.

In joint work with Emmanuel Dror-Farjoun, we have developed a spectral sequence aimed at computing these groups in a wide variety of settings.
Spectral sequences reveal and conceal. On the one hand, there are long lists of spectacular results in topology and algebra that have
been obtained by spectral sequence techniques (cf. \cite{McCleary}). On the other hand, spectral sequence cognoscenti  will testify  that there is depressingly little known in general
about differentials in spectral sequences, so that reducing a problem to a \lq\lq spectral sequence calculation\rq\rq may not in fact solve the problem at all.
 In order to put the results of this paper in context, we review briefly what is known about differentials in this genre of spectral sequence.

The classical Atiyah - Hirzebruch spectral sequence \cite{AH}  takes the form
\[
E_2 = H^*(X ; K^*(pt))  \Longrightarrow  K^*(X)  .
\]
Atiyah and Hirzebruch noted   the following:
\begin{enumerate}
\item $d_{2n}  = 0 $ for all $n$ because $K^{odd}(pt) = 0$.
\item $d_3$ is associated with the integral Steenrod operation
\[
Sq^3 : H^k(X; \ZZ ) \to H^{k+3}(X; \ZZ ).
\]
\item For $k \geq 2$, each $d_k $ takes values in the torsion subgroup of $E_k$ and hence the spectral sequence collapses rationally:
\[
E_2 \otimes \QQ   \cong E_\infty \otimes \QQ  .
\]
\end{enumerate}

Arlettaz \cite{Arlettaz} gives explicit integers governing the order of the torsion subgroups.
In addition, there are always \lq\lq dimension\rq\rq\, arguments in particular cases. For example, if $H^*(X; \ZZ ) = 0$ in all 
odd degrees, then $E_2 = 0$ in odd total degree,
$d_j = 0$ for all $j \geq 2$, and so $ E_2    \cong E_\infty $.

More generally, if $h^*$ is a generalized cohomology theory then there is a well-known spectral sequence
\[
E_2 = H^*(X ; h^*(pt))  \Longrightarrow  h^*(X)  .
\]
In another direction there is the classical Federer spectral sequence \cite{Federer}
\[
E_2 = H^*(X ; \pi _*(Y))  \Longrightarrow  H^*(F(X,Y)),
\]
where $F(X,Y)$ denotes the function space of maps from $X$ to $Y$ with the compact-open topology. Sam Smith \cite{Smith} shows that in the context of Quillen minimal models, differentials are related to Whitehead products.
(Our results will echo this result in the integral situation.)

Moving to twisted $K$-theory $K_\Delta ^*(X)$ associated to a principal bundle,    the
bundle is classified by its Dixmier-Douady invariant $\Delta $
and the spectral sequence takes the form
\[
E_2 = H^*(X ; K^*(pt))  \Longrightarrow  K_\Delta ^*(X)
\]
with the $d_3$ differential related to the Dixmier-Douady invariant by the result of J. Rosenberg \cite{Rosenberg}.
  Atiyah and Segal  (\cite{AS} Prop. 7.5)
  show that if the base space is a compact manifold then in the associated spectral sequence
   \[
E_2 = H^*(X ; K^*(pt))\otimes \RR   \Longrightarrow  K_\Delta ^*(X)\otimes \RR ,
\]
  all higher $(j \geq 4) $ differentials are  given by  Massey products. They point
out that this implies that all higher differentials vanish (over the reals) when the base space is a compact K\"ahler manifold, by
the deep result of \cite{DGMS}.

Let $\GLo B$ denote the path component of the identity of $\GL B$. Similarly, for $C^*$-algebras, let $ \Uo   B$ denote
the path component of the identity of the unitary group $ \U    B$. Recall that the inclusions $ \Uo   B \subset \GLo B$ and $ \U    B \subseteq \GL B$ are deformation 
retractions and hence homotopy equivalences.

We consider the case of a fibre bundle
over a sphere in order to isolate a new type of differential. Here are our primary results. Let $\PP $ denote a subring of the rational numbers.
 (We allow the cases $\PP = \ZZ$ and $\PP = \QQ$ as well as intermediate rings.)
We define a    {\emph{bundle of $C^*$-algebras}}  precisely in \S 2.

\begin{bigthm}\label{T:maintheoremhomotopy}
Suppose that     \[
   {\zeta \otimes B} : \qquad  P  \times _G B  \to S^k
   \]
    is a    bundle of $C^*$-algebras over $S^k$ with $k \geq 2$. Let
 $A_{\zeta \otimes B} $ denote the associated $C^*$-algebra of continuous sections.  Then there is a long exact Wang\footnote{
 Wang's original 1949 paper \cite{Wang} gave a direct and elementary proof of a  homology version of this sequence.
 Serre (\cite{Serre} p. 471)  put the result  into the spectral sequence setting. Since then it has appeared in
 many contexts.}
sequence
\[
\dots \to
\pi _n( \Uo   A_{\zeta \otimes B})\otimes\PP
\overset{r}\to
\pi _n( \Uo   B)\otimes\PP
\overset{d_k}\longrightarrow
\pi _{n+k-1}( \Uo   B)\otimes\PP
\overset{s}\to
\pi _{n-1}( \Uo   A_{\zeta \otimes B})\otimes\PP
\to\dots
\]
The differential $d_k$ is given by
\[
 d_k(a)  = -\,g[ \kappa, a ]
 \]
 where $g$ is the generator of $H^k(S^k;\ZZ )$ and $[ \kappa, a ]$ is the \esp\,    with the map $\kappa : S^{k-1} \to G$ that classifies the principal bundle.
\end{bigthm}

The enhanced Samelson product referred to is a generalization of the classical Samelson product that we explain in  \S 6.

Passing to limits, we also obtain the analogous sequence at the level of $K$-theory:

\begin{bigthm}\label{T:maintheoremk}
Suppose that  
\[
{\zeta \otimes B} :\qquad    P  \times _G B  \to S^k
\]
 is a    bundle of $C^*$-algebras over $S^k$ with $k \geq 2$. Let
 $A_{\zeta \otimes B} $ denote the associated $C^*$-algebra of continuous sections.  Then there is a long exact sequence
\[
\dots \to
K_n(A_{\zeta \otimes B}) \otimes\PP
\overset{\rho}\to
K_n(B)   \otimes\PP
\overset{d_k}\longrightarrow
K _{n+k-1}(B)\otimes\PP
\overset{\sigma}\to
K _{n-1}( A_{\zeta \otimes B})\otimes\PP
\to\dots
\]
When $k = 3$ then $d_3 $ is given by multiplication by $- \Delta _\zeta \beta $ where $\Delta _\zeta $ is the Dixmier-Douady integer and $\beta $ is Bott
periodicity.

\end{bigthm}

 The paper is organized as follows. First we state the general spectral sequence results of \cite{FS}, we specialize them to the
 case when the base is a sphere, and we derive the homological algebra version of the Wang sequence. We then make a preliminary identification of the differential. Then we make a detour to classical homotopy theory to   define the enhanced Samelson product. Next, we put everything together to establish Theorems A and B.

 Finally, we illustrate our results with a close analysis of the invariants associated to 
the $C^*$-algebra of sections of the bundle
\[
\zeta \otimes M_2 : \qquad   S^7 \times _{S^3} M_2  \to S^4
\]
 constructed from the Hopf bundle $\zeta : S^7 \to S^4$ and by the  conjugation action of $S^3$ on $M_2 = M_2(\CC )$. 
  We explicitly compute the homotopy groups 
  $\pi _n   {    (\Uo  A_{\zeta\otimes M_2}    )   }$
  for $n \leq 8$ and
 contrast these results with the computation of the rational homotopy groups $\pi _*(\Uo  A_{\zeta\otimes M_2 }  )\otimes\QQ$  and with 
 the topological $K$-theory groups 
 $K_*(A_{\zeta\otimes M_2} )$
 by explicit examination of the relevant groups and the maps
 \[
 \begin{CD}
 \pi _n( \Uo  A_{\zeta\otimes M_2 }  )     @>>>       \pi _n( \Uo  A_{\zeta\otimes M_2 }  )\otimes\QQ  \\
 @VVV    \\
 K_{n+1}(A_{\zeta\otimes M_2} ).
 \end{CD}
 \]

 \vglue .1in
The following table summarizes our calculation.
 \vglue .2in

\begin{centerline}{\bf{The homotopy, rational homotopy and K-theory of $A_{\zeta\otimes M_2}$ in low degrees}}
\end{centerline}
\vglue .1in
\begin{center}
\begin{tabular}{  |l  | c | c  | c | }
\hline
n    &   $ \pi _n( \Uo  A_{\zeta\otimes M_2}) $    &   $ \pi _n( \Uo  A_{\zeta\otimes M_2} )\otimes\QQ  $  &  $ K_{n+1}(A_{\zeta\otimes M_2})$   \\   \hline\hline
  &  &  &  \\ \hline
1    &   $\ZZ \oplus \ZZ/2$   &  $\QQ$   &  $\ZZ$   \\   \hline
2    &   0   &   0 &  0   \\   \hline
3    &   $\ZZ$   &   $\QQ $  &  $\ZZ$     \\   \hline
4    &   0   &   0 &  0   \\   \hline
 5   &   0   &   0 &  $\ZZ $   \\   \hline
 6    &   $\ZZ /60$    &   0 &  0   \\   \hline
 7    &   $\ZZ /4 $  or  $(\ZZ/2)^2 $  &   0 &  $\ZZ $  \\   \hline
 8    &   $\ZZ/4 \oplus \ZZ/2$ or $(\ZZ/2)^3$     &   0 &  0   \\   \hline

\end{tabular}
\end{center}
\vglue .2in

 It is a pleasure to thank my collaborators and colleagues  Emmanuel Dror-Farjoun,  Dan Isaksen, John Klein, Bert Schreiber,
 and Sam Smith for continued assistance and to acknowledge the insight that I received from the work of C. Wockel.

\section{The spectral sequences}
We recall for reference the main results of Farjoun-Schochet \cite{FS}.

Suppose that $X$ is a finite dimensional compact metric space   and  $ \zeta : P  \to X$ is a standard\footnote{ A principal $G$-bundle is  {\emph{standard}}   if $X$ is a finite complex or if the bundle is a pullback of a principal $G$-bundle over some $CW$-complex.  See \cite{FS} for details.}
principal $G$ bundle for some
   topological group $G$ that acts continuously on a Banach algebra $B$ via $\alpha : G \to Aut(B)$.  Let
   \[
   \zeta\otimes B :\qquad   P  \times _G B \to X
   \]
     be the associated fibre bundle. (We refer to this set-up as a {\emph{standard bundle of Banach algebras.}})
  Let
 \[
 A_{\zeta \otimes B} = \Gamma (X, P  \times _G B)
 \]
denote  the set of continuous sections of the bundle with pointwise operations. This has a natural structure of a   Banach algebra, and if $B$ is a $C^*$-algebra then so is  $A_{\zeta \otimes B}$.   If $B$ is unital, then
$A_{\zeta \otimes B} $ is also unital, with identity the canonical
 section that to each point $x \in X$ assigns the identity in $(P  \times _G B)_x$. 

 We are interested in $\GL A_{\zeta \otimes B} $, the group of invertible elements
 in
 $A_{\zeta \otimes B}$. (If $A_{\zeta \otimes B} $ is not unital, then we understand this to mean the kernel of the natural map
 $\GL(A_{\zeta \otimes B} ^+) \to \GL(\CC )$.)
This is a space of the homotopy type of a CW-complex,  second countable if $B$ is separable. It may have many
(homeomorphic) path components; let  $\GLo A_{\zeta \otimes B} $
denote the path component of the identity.

If $A$ is a unital  $C^\ast$ -algebra then denote   the group of unitary elements of $A$ by $ \U   A$ and its identity path component by $ \Uo   A$.  If it is not unital then we define $ \U    A$ to be the kernel of the natural map $ \U   (A^+) \to  \U   \CC $ and similarly for $ \Uo   A$.
The inclusion
$ \U    A \to \GL A $ is a   homotopy equivalence.  

Let $\PP $ denote a subring of the rational numbers. (We allow the cases $\PP = \ZZ$ and $\PP = \QQ$ as well as intermediate rings.)

 \begin{Thm}\label{T:theoremA}
  Suppose that $X$ is a finite dimensional compact metric space  and  that
   ${\zeta \otimes B} :   P  \times _G B  \to X    $ is a  standard   bundle of Banach algebras. Let
 $A_{\zeta \otimes B} $ denote the associated  algebra of sections.
Then:
 \begin{enumerate}
\item There is a second quadrant spectral sequence
  converging to
  \[
  \pi _*(\GLo A_{\zeta \otimes B} )\otimes \PP
  \]
 with
 \[
 E^2_{-p,q}  \cong H^p(X ; \pi _q(\GLo B)\otimes \PP )
 \]
 and
 \[
 d^r : E^r_{-p,q} \longrightarrow E^r_{-p-r, q+r-1}  .
 \]
 \item If
 $X$ has dimension at most $n$, then $E^{n+1} = E^\infty $.
  \item The spectral sequence is natural with respect to pullback diagrams
  \[
  \begin{CD}
  f^*P \times _G B  @>{f \times 1}>>   P \times _G B  \\
  @VVf^*{\zeta \otimes B} V  @VV{\zeta \otimes B} V    \\
   X' @>f>>   X
 \end{CD}
 \]
 \vglue .1in
 and associated map $f^*:   A_{\zeta \otimes B} \longrightarrow  A_{f^*{\zeta \otimes B} } $.
 \item The spectral sequence is natural with respect to $G$-equivariant maps
 \[
 \phi : B \to B'
 \]
 of   Banach algebras.
 \end{enumerate}
\end{Thm}

Generally, this spectral sequence does not collapse, even rationally.

Note that in many cases of interest, for instance    $B = M_n(\CC ) $,     the  groups $\pi _*(\GLo B)$ are unknown, and so the
integral version of the spectral sequence cannot be used directly to compute $\pi _*(\GLo A_{\zeta \otimes B} )$. However, frequently the groups
$\pi _*(\GLo B)\otimes\QQ$ {\emph{are}} known and hence the rational form of the spectral sequence will be practical.

Using the version of Bott periodicity established by R. Wood \cite{Wood} and M. Karoubi \cite{Kar} and taking limits of spectral sequences, we derive the following.

\begin{Thm}\label{T:theoremB}
Suppose that $X$ is a finite dimensional compact metric space,
 $B$ is a    Banach algebra, and ${\zeta \otimes B} :     P  \times _G B  \to X    $ is a  standard bundle of Banach algebras. Then
 there is a second quadrant spectral sequence
 \[
 E^2_{-p,q}   \cong H^p(X ; K _{q+1}(B)\otimes \PP )   \Longrightarrow    K _{*+1}(A_\zeta )\otimes \PP
 \]
which is the direct limit over $t$ of the corresponding spectral sequences converging to
\[
 \pi _*(\GLo (A_{\zeta \otimes B \otimes M_t} )\otimes \PP  .
 \]
If $X$ has dimension at most $n$ then $E^{n+1} = E^\infty $.
\end{Thm}

This result   is due to J. Rosenberg \cite{Rosenberg} when $B = \KK $ and $A_{\zeta\otimes \KK} $ is a continuous trace $C^*$-algebra over a finite complex $X$.

\section{The spectral sequences when $X$ is a sphere}

Suppose that $X = S^k$ with $k \geq 2$. Then several things simplify radically. First of all,  bundles are automatically standard
and so we will simply say \emph{bundle of Banach algebras} 
in this setting.
Second, $X$ is simply connected and hence
the local coefficients in the spectral sequences trivialize.  Third, the $E^2$ term vanishes except in columns $p = 0$ and $p = -k$.
This implies that $E^2 = E^k$ and $E^{k+1} = E^\infty $ in the spectral sequences, so that the only possible non-zero higher differential
is $d_k$. Combining these elementary observations we have the following versions of Theorems \ref{T:theoremA} and
\ref{T:theoremB}.

\begin{Thm}\label{T:theoremA'}
  Suppose that
   ${\zeta \otimes B} :     P  \times _G B  \to S^k    $ is a    bundle of Banach algebras with $k \geq 2$. Let
 $A_{\zeta \otimes B} $ denote the associated section algebra.
Then
  there is a second quadrant spectral sequence
  converging to
  \[
  \pi _*(\GLo A_{\zeta \otimes B} )\otimes \PP
  \]
 with $E^2$ = 0 except for
 \[
 E^2_{0,q} =  E^k_{0,q}   \cong H^0(S^k ; \ZZ )\otimes  \pi _q(\GLo B)\otimes \PP
 \]
 \[
 E^2_{-k,q} =  E^k_{-k,q}   \cong H^k(S^k ; \ZZ )\otimes  \pi _q(\GLo B)\otimes \PP
 \]
 and the only higher possibly non-zero differential is
 \[
 d^k : E^k_{0,q} \longrightarrow E^k_{-k, q+k-1}  .
 \]
 Thus $E^{k+1} = E^\infty $.

\end{Thm}

\begin{Thm}\label{T:theoremB'} Suppose that
   ${\zeta \otimes B} :   P  \times _G B  \to S^k    $ is a    bundle of Banach algebras with $k \geq 2$.
Then there is a second quadrant spectral sequence converging to  $K _{*+1}(A_{\zeta \otimes B} )\otimes \PP$
with $E^2$ = 0 except for
 \[
 E^2_{0,q} =  E^k_{0,q}   \cong H^0(S^k ; \ZZ )\otimes K _{q+1}(B)\otimes \PP
 \]
 \[
 E^2_{-k,q} =  E^k_{-k,q}   \cong H^k(S^k ; \ZZ )\otimes  K _{q+1}(B)\otimes \PP
 \]
  and the only possibly non-zero higher differential is
 \[
 d^k : E^k_{0,q} \longrightarrow E^k_{-k, q+k-1}  .
 \]
 \item Thus $E^{k+1} = E^\infty $.
\end{Thm}

\section {Deriving the Wang sequence}

We may rephrase the conclusion of Theorem \ref{T:theoremA'} as asserting the existence of a long exact sequence
\[
0  \to E^\infty _{0,q}   \to  E^k_{0,q}  \overset {d^k}\longrightarrow E^k_{-k, q+k-1}     \to E^\infty _{-k, q+k-1}  \to 0
\]
and after identifications we obtain the exact sequence
\[
0  \to E^\infty _{0,q}   \to   \pi _q(\GLo B)\otimes \PP   \overset {d^k}\longrightarrow  \pi _{q+k-1}(\GLo B)\otimes \PP     \to E^\infty _{-k, q+r-1}  \to 0.
\tag{*}
\]

On the other hand, the filtration that creates the spectral sequence   comes  from the cell filtration
of $X = S^k$ and hence simplifies dramatically to become
\[
0\longrightarrow  E^\infty _{-k,n+k}   \longrightarrow     \pi _n(\GLo A_{\zeta \otimes B})\otimes\PP \longrightarrow  E^\infty _{0,n}  \longrightarrow  0
\tag{**}
\]
with
\[
E^\infty _{-k,n+k} \cong    H^k(S^k)\otimes \pi _{n+k}(\GLo B)   \otimes \PP
\]
and
\[
E^\infty _{0,n}   \cong  H^0(S^k)\otimes \pi _{n}(\GLo B) \otimes \PP.
\]
 \vglue .1in

Splice the two sequences $\{*\} $  and $\{**\}$ together as follows.
Splicing the sequences at $E^\infty _{0,n}$ gives the composite
\[
r:
\pi _n(\GLo A_{\zeta \otimes B})\otimes\PP \to E^\infty _{0,n}
\to E^k _{0,n}   \cong  \pi _n(\GLo B)\otimes\PP .
\]
It is easy to see that the map $r$ is induced by the evaluation map
\[
r:
 \GLo A_{\zeta \otimes B} \longrightarrow
 \GLo B
\]
that takes a section and restricts it to the basepoint $x_0 \in S^k$.

Splicing the sequence at $E^\infty _{-k,n+k-1}$ gives the composite
\[
s :   \pi _{n+k-1}(\GLo B)\otimes\PP  \cong
E^k _{-k,n+k-1}  \to
E^\infty  _{-k,n+k-1}  \to  \pi _{n-1}(\GLo A_{\zeta \otimes B})\otimes\PP
\]
where the map $s $ corresponds to the inclusion of a pointed section into the space of all sections.

We obtain the following generalization of the Wang sequence  \cite{MacLane}.

\begin{Thm}\label{T:theoremsphere}
Suppose that $X = S^k$ with $k \geq 2$    and  that
\[
   {\zeta \otimes B} :  \qquad   P  \times _G B  \to S^k    
   \]
    is a     bundle of Banach algebras. Let
 $A_{\zeta \otimes B} $ denote the associated section algebra.  Then there is a long exact sequence
\[
\dots \to
\pi _n(\GLo A_{\zeta \otimes B})\otimes\PP
\overset{r}\to
\pi _n(\GLo B)\otimes\PP
\overset{d_k}\longrightarrow
\pi _{n+k-1}(\GLo B)\otimes\PP
\overset{s}\to
\pi _{n-1}(\GLo A_{\zeta \otimes B})\otimes\PP
\to\dots
\]

\end{Thm}

Passing to limits, we also obtain the analogous sequence at the level of $K$-theory:

\begin{Thm}\label{Ktheorem}
Suppose that $X = S^k$ with $k \geq 2$    and  that
  \[
  {\zeta \otimes B} : \qquad    P  \times _G B  \to S^k    
  \]
   is a    bundle of Banach algebras. Let
 $A_{\zeta \otimes B} $ denote the associated section algebra.  Then there is a long exact sequence
\[
\dots \to
K_n(A_{\zeta \otimes B}) \otimes\PP
\overset{r}\to
K_n(B)   \otimes\PP
\overset{d_k}\longrightarrow
K _{n+k-1}(B)\otimes\PP
\overset{s}\to
K _{n-1}( A_{\zeta \otimes B})\otimes\PP
\to\dots
\]

\end{Thm}

\begin{Rem}
This result agrees with the result of Rosenberg \cite{Rosenberg} on continuous trace algebras over $S^3$. He shows there that
  if $d_3$ is an isomorphism then $K_*(A_{\zeta\otimes\KK}) = 0$. If $d_3 = 0$ then $K_0(A_{\zeta\otimes\KK}) = \ZZ$ and $K_1(A_{\zeta\otimes\KK}) = 0$.
If $d_3 $ is multiplication by $s \neq 0,\pm1 $ then  $K_0(A_{\zeta\otimes\KK}) = 0$ and $K_1(A_{\zeta\otimes\KK}) = \ZZ /s$.
\end{Rem}

To illustrate the use of the Wang sequence for those not so familiar with spectral sequence
arguments, we calculate a simple example.

\begin{Thm}\label{T:easythom} Suppose that $\zeta\otimes B :  P \times _G B \to S^k$ is a bundle of $C^*$-algebras. Assume
that  $k$ is even and that $K_1(B) = 0$. Then there is a short exact sequence
\[
0  \to K_k(B) \overset{s}\longrightarrow K_0(A_{\zeta\otimes B} ) \overset{r}\longrightarrow 
K_0(B) \to 0 .
\]
and $K_1(A_{\zeta\otimes B} ) = 0$.

\end{Thm}

\begin{proof} The differential in the Wang sequence is a map $d_k : K_n(B) \to K_{n+k - 1}(B)$ 
and since $k$ is even it changes parity. The fact that $K_1(B) = 0$ then implies that $d_k$ is 
always the zero map. Thus the long exact  Wang sequence degenerates as shown. 
\end{proof}

\begin{Rem} We believe that Theorem \ref{T:easythom} is a hint at a non-commutative Thom 
isomorphism theorem, generalizing the classical Thom isomorphism theorem that 
for a complex vector bundle $V \to X$ relates $K^*(X)$ with  the $K$-theory 
of the Thom space of the bundle.  The map $r$ corresponds to the zero section of the bundle, and 
so it is not unreasonable to define 
\[
\tilde{K}_*(A_{\zeta\otimes B} ) \cong Ker \big[ K_*(A_{\zeta\otimes B} ) \overset{r}\longrightarrow 
K_*(B) \big]
\]
so that there is a Thom-type isomorphism
\[
 K_{k+*}(B) \overset{\cong}\longrightarrow \tilde{K}_*(A_{\zeta\otimes B} )
\]
induced by the map $s$.  Note that the map $s$ is locally in effect the product with the 
$K$-theory fundamental class of the even-dimensional sphere $S^k$ so this has the right 
flavor as well.

\end{Rem}

\section  {Identifying the differential}

In order to   identify the unknown differential in these theorems, we must look   at the exact couple that
gives rise to the spectral sequence as constructed in \cite{FS} \S 4. Suppose that $X$ is a finite complex. The space of invertible sections of the bundle
\[
P\times _G \GLo B \longrightarrow X
\]
  is filtered up to homotopy by a   descending filtration
\[
\dots   \FF _{p+1}X  \hookrightarrow    \FF _pX  \dots
\]
(See \cite{FS} for details.)
 The resulting exact couple is given by
\[
D^1 _{-p,q}  \cong \pi _{q-p}(\FF _pX )\otimes \PP
\]
and
\[
 E^1 _{-p,q} = \pi _{q-p} ((\FF _{p}X)/(\FF _{p+1}X)) \otimes \PP .
 \]
 The structural maps are given as follows:
 \begin{enumerate}
 \item The map $i^1 : D^1_{p,q} \to D^1_{p+1, q-1} $ is given by the natural map induced by the filtration:
\[
 \begin{CD}
 D^1_{-p,q} @>i^1>>   D^1_{-p+1, q-1} \\
@VV\cong V                  @VV\cong V  \\
\pi _{q-p}(\FF _pX ) @>>>   \pi _{q-p}(\FF _{p-1}X )
 \end{CD}
 \]
 \item  The map
 $j^1 : D^1_{-p,q} \to E^1_{-p, q} $ is given by
 \[
 \begin{CD}
 D^1_{-p,q} @>j^1>>   E^1_{-p, q} \\
@VV\cong V                  @VV\cong V  \\
\pi _{q-p}(\FF _pX ) @>>>   \pi _{q-p} ( (\FF _{p}X)/(\FF _{p+1}X)) \otimes \PP
 \end{CD}
 \]
 \item
 The map $\delta ^1 : E^1_{p,q} \to D^1_{p-1, q}$
 is given by
 \[
 \begin{CD}
 E^1_{-p,q} @>\delta ^1>>   D^1_{-p -1, q} \\
@VV\cong V                  @VV\cong V  \\
\pi _{q-p} ( (\FF _{p}X)/(\FF _{p+1}X)) \otimes \PP   @>{\delta}  >> \pi _{q-p -1} ( (\FF _{p+1}X)) \otimes \PP
 \end{CD}
 \]

 \end{enumerate}

 with differential
 \[
 d^1 = j^1\delta ^1 : E^1_{-p,q} \longrightarrow E^1_{-p-1,q}.
 \]
 We may identify the $E^1$ term by noting that
 \[
  (\FF _{p}X)/(\FF _{p+1}X)  \simeq     \vee _\alpha  S^p
 \]
 a wedge of spheres, and hence
 \[
 E^1_{-p,q} \cong \pi _{q-p} (\Gamma (\vee _\alpha  S^p, (\GLo A_{\zeta \otimes B} )|_{S^p} )) \otimes \PP
  \cong \pi _{q-p} (F_*(\vee _\alpha  S^p , \GLo B ))\otimes \PP  \cong
 \]
 \[
 \cong \oplus _\alpha \pi _{q-p}(\Omega ^p \GLo B)\otimes \PP  \cong \oplus _\alpha \pi _q(\GLo B) \otimes \PP \cong
  C^p(X ; \pi _q(\GLo B)\otimes \PP)
  \]
   so that
  \[
    E^1_{-p,q} \cong    C^p(X ; \pi _q(\GLo B)\otimes \PP),
  \]
  the cellular cochains of $X$ with coefficients in $\pi _q(\GLo B)\otimes \PP$.
  The $d^1$ differential is  the usual cellular differential  and so
   \[
  E^2 _{-p,q} \cong \check{H}^p(X ; \pi _q(\GLo B)\otimes \PP).
  \]
  However, when $X = S^k$ then the matter becomes a lot simpler- the $E^2$ term vanishes except for $p = 0, k$. The differentials
  $d^s$ vanish  for $s < k$. Internally, the derived exact couples have the property that the maps
  \[
  D^s_{-u,v} \overset{i^s}\longrightarrow
 D^s_{-u + 1,v-1 } \overset{i^s}\longrightarrow  \dots \longrightarrow
  D^s_{-1, v - u +1 }
\]
are isomorphisms for $s \leq  k$. Thus we may identify the $d^k$ differential as the composite

\[
 E^k_{0,q}
  \overset{\delta ^k}\longrightarrow
  D^k_{-1, q  }
\overset{\cong}\longleftarrow  \dots \overset{\cong}\longleftarrow
 D^k_{-k, q+k - 1 }
 \overset{ j^k  }\longrightarrow
  E^k_{-k, q+k - 1 }
\]

We summarize:

\begin{Pro}\label{T:differential}  In the case of Theorem \ref{T:theoremsphere}  where $X = S^k$, the $d^k$ differential is given as the composite
 \[
 E^k_{0,q}
  \overset{\delta ^k}\longrightarrow
  D^k_{-1, q  }
\overset{\cong}\longleftarrow  \dots \overset{\cong}\longleftarrow
 D^k_{-k, q+k - 1 }
 \overset{ j^k  }\longrightarrow
  E^k_{-k, q+k - 1 }
\]
where
  \[
 E^2_{0,q} =  E^k_{0,q}   \cong H^0(S^k ; \ZZ )\otimes  \pi _q(\GLo B)\otimes \PP
 \]
 \[
 E^2_{-k,q+k - 1} =   E^k_{-k,q+k - 1}   \cong H^k(S^k ; \ZZ )\otimes  \pi _{q+k - 1}(\GLo B)\otimes \PP  .
 \]
\end{Pro}\qed
\vglue .1in
So what is this map?  Its identification requires a detour.  We must   generalize the classical Samelson product.

\section{Enhanced Samelson products }

Let $G$ be a topological group and let $Y_1$ and $Y_2$ be topological spaces with distinguished 
basepoint. We take the identity as the basepoint for topological groups. Let $[\,\,,\,\,]$ 
denote based homotopy classes of maps.  The traditional Samelson product  (cf. \cite{Samelson}, \cite{Whitehead} p. 467, \cite{Neis}  \S 6.3) is a pairing
\[
[\, \,\,,\,\,\, ] \, : \qquad   [Y_1, G ] \times [Y_2 ,G ] \longrightarrow [Y_1 \wedge Y_2  ,G  ]
\]
defined by
\[
[[\phi ] , [\psi] ] = [\eta ] \qquad\text{where}\qquad \eta (w\wedge  y) = \phi (w)\psi (y)\phi(w)^{-1}\psi(y)^{-1} .
\]
If $Y_1 = S^r$ and $Y_2  = S^s$ this induces a pairing
\[
[\,\,\,  ,   \,\,\,   ] : \qquad  \pi _r(G ) \times \pi _s (G ) \longrightarrow \pi _{r+s}(G ).
\]
We wish to generalize this construction to our context.

\begin{Def}
Suppose  that $B$ is a Banach algebra with a continuous group action given by a map $\alpha :G  \to Aut(B)$. We
define the {\emph{\esp }}
\[
[Y_1,G  ] \times [Y_2,  \GL B ] \longrightarrow [Y_1 \wedge Y_2  , \GL B]
\]
by
\[
[[\phi ] , [\psi] ] = [\eta ] \qquad{\text{where}}\qquad    \eta (w\wedge  y) = \alpha_{\phi (w)}(\psi(y))    \psi(y)^{-1}.
\]
\vglue .1in
Taking   $Y_1 = S^r$ and $Y_2  = S^s$  gives an  {\emph{\esp }}
 \[
[\, \,\,,\,\,\, ] \, : \,      \pi _r(G ) \times \pi _s (\GL B ) \longrightarrow \pi _{r+s}(\GL B  ).
\]
If $B$ is a $C^*$-algebra then the same formula induces 
 an  {\emph{\esp }}
 \[
[\, \,\,,\,\,\, ] \, : \,      \pi _r(G ) \times \pi _s ( \Uo   B ) \longrightarrow \pi _{r+s}( \Uo   B  ).
\]

It is elementary to show that a morphism of $G$-algebra $B \to B'$ induces a commuting 
diagram 
\[
\begin{CD}
  \pi _r(G ) \times \pi _s (\GL B )  @>{[\, \,\,,\,\,\, ]}>> \pi _{r+s}(\GL B  ) \\
 @VVV   @VVV     \\
  \pi _r(G ) \times \pi _s (\GL B' )  @>{[\, \,\,,\,\,\, ]}>>    \pi _{r+s}(\GL B'  ) \\ 
 \end{CD} 
\]
and hence induces a pairing

 \[
[\, \,\,,\,\,\, ] \, : \,      \pi _r(G ) \times K_s(B) \longrightarrow K_{s + r}(B)  .
\]

\end{Def}

\begin{Rem}
If the action of $G $ on $B$ is inner then we write the action of $G$ on $\GL B$ as
\[
(g, b) \longrightarrow  gbg^{-1}.
\]
Then
\[
 \alpha_{\phi (w)}(\psi(y))   =  \phi (w)\psi (y)\phi(w)^{-1}
\]
and hence
\[
[[\phi ] , [\psi] ] = [\eta ] \qquad{\text{where}}\qquad    \eta (w\wedge  y) = \phi (w)\psi (y)\phi(w)^{-1}\psi(y)^{-1}
\]
which is the traditional Samelson formula. So the \esp\, is a true generalization of the classical Samelson product..
\end{Rem}

\begin{Rem}
If $B$ is a $C^*$-algebra and $G $ is locally compact then $(B,G )$ form what G. Pedersen \cite{Pedersen} calls a
\emph{$C^*$-dynamical system}. Pedersen shows (\cite{Pedersen}, p. 257) that there is a Hilbert space $\HH$ upon which $G$ acts by the regular representation $\lambda $
and a faithful covariant representation $\rho $  of $( B,G )$ of the dynamical system. Thus up to isomorphism we may replace $(B,G )$ by $(\rho(B),G )$.
Then
\[
\rho (\alpha _g(b) ) = \lambda _g \rho(b) \lambda _g^{-1}.
\]
So in this case too the \esp  \,  reduces down to a commutator of the form $U_gb U_g^{-1} b^{-1} $ as well, even though $U_g$ is not in $\GL B$.
\end{Rem}

\section{Identifying the differential more precisely}

We introduce some notation in order to analyze the situation over spheres. Let $G$ be a topological group and let
  $\zeta : P \to S^k$ denote a principal $G$-bundle. It is classified by its clutching map $\kappa : S^{k-1} \to G $ which we realize explicitly as follows.

 Let $D^n$ denote the $n$-ball, regarded  as the cone on $S^{n-1}:$
\[
D^n = [0,1] \times S^{n-1} / \{0\} \times S^{n-1}
\]
and we write $(t, x) \to tx$.    We decompose the base space $S^k$ as the disjoint union $S^k = H^+  \cup H^- $
of upper and lower closed hemispheres, with equator $S^{k-1} = H^+  \cap H^- $. The restriction of the principal bundle $\zeta : P \to S^k$
to each closed hemisphere trivializes, and so there are sections
\[
\sigma ^{\pm}  : H^\pm \to H^\pm \times G  \qquad\qquad \sigma ^{\pm}(x) = (x, s^\pm (x) )
\]
 and  a clutching map
$\kappa : S^{k-1} \to G $    satisfying
\[
s^+(x) = \kappa (x)s^-(x)\kappa (x) ^{-1}   \quad  \forall \,x \,\in S^{k-1}
\]
that determine the principal bundle $\zeta $  up to equivalence. 
The triviality of $\zeta $ over each hemisphere implies that the associated bundle
\[
P \times _G \GLo B \to S^k
\]
 is also trivial over each
hemisphere. Thus we may alternately describe  $\GLo(A_{\zeta \otimes B} )$ as
 \[
 \GLo ' A_{\zeta \otimes B} = \{ (s^+, s^- ) : H^+ \sqcup H^- \to G :
  s^+(x) = \kappa (x)s^-(x)\kappa (x) ^{-1}
 \quad  \forall \,x \,\in S^{k-1} \}
\]
where we identify $\GLo(A_{\zeta \otimes B} )  \cong \GLo '(A_{\zeta \otimes B} ) $ by
\[
s  \to  ( s \sigma ^+, s\sigma ^- ).
\]
  Taking
the south pole $x_o$ as basepoint of $S^k$, the evaluation map
\[
r :  \Gamma (   S^k,         P\times _G \GLo B )  \longrightarrow \GLo B
\]
 is given in this picture by
$  r(s^+, s^-) =  s^-(x_o) $.

Define the set of based maps  
\[
F_\bullet (S^k, \GLo B ) = \{ s \in F(S^k, \GLo B ) : s(x_o) = e \}
\]
and similarly for based spaces of sections $\GLo _\bullet  $ and $ \Uo   _\bullet $. 

The following proposition would seem to be folklore.

 \begin{Pro}
  There is a natural identification
\[
\pi _{n}(\GLo_\bullet  A_{\zeta \otimes B}) \cong \pi _{n} (F_\bullet (S^k, \GLo B)) \cong
 \pi _{n+ k } (  \GLo B) .
 \]
 \end{Pro}

\begin{proof} Let
 \[
 \GLo ''  A_{\zeta \otimes B} = \{ (s^+, s^- )  \in  \GLo '  A_{\zeta \otimes B}   : s^-(x) = e\,\,\, \forall x \in H^- \}.
 \]
 It is an exercise (cf. Wockel \cite{Wockelthesis} Lemma 4.1.6 ) to show that the natural inclusion
 \[
  \GLo ''  A_{\zeta \otimes B}            \longrightarrow   \GLo '  A_{\zeta \otimes B}
 \]
 is a homotopy equivalence.  But then it is easy to see that
 \[
  \GLo ''  A_{\zeta \otimes B} = \{ s^+ : H^+ \to \GLo B : s^+(x) = e \,\,\,\forall x \in \partial(H^+) \} \cong F_\bullet (S^k,\GLo B)
  \]
  so the proposition is immediate.
  \end{proof}

The first two parts of the following result are due to Thomsen (\cite{Thomsen}, Theorem 1.9). 
The third part generalizes  Wockel \cite{Wockel}, Theorem 2.3 and we have adapted his proof as well.

\begin{Thm}\label{T:Wockel}   Let $\zeta : P \to S^k$
be a principal $G$-bundle with clutching map $\kappa  : S^{k-1} \to G $.   Let $B$ be a $C^*$-algebra upon 
which $G$ acts  and 
let 
\[
A_{\zeta \otimes B}  = \Gamma (S^k, P\times _G B   )
\]
 denote the associated $C^*$-algebra. Then:
\begin{enumerate}
\item There is an associated exact sequence of topological groups 
 \[
 (\diamond )\qquad\qquad       \Uo  _\bullet A_{\zeta \otimes B}      \longrightarrow   \Uo   A_{\zeta \otimes B} \overset{r}
 \longrightarrow  \Uo   B .      
 \]
\item The evaluation map $r$ admits continuous local sections, and so $(\diamond )$ is a fibre bundle over a 
paracompact space, hence a fibration. 
\item
 Let $\partial $ denote the boundary homomorphism in the long exact homotopy sequence associated to the
evaluation fibration
and let $\delta _n$ denote the composition
\[
\pi _n( \Uo   B) \overset{\partial}\longrightarrow  \pi _{n-1}(  \Uo  _\bullet  A_{\zeta \otimes B}) \cong \pi _{n-1} (F_\bullet (S^k,  \Uo   B) \cong
 \pi _{n+ k -1} (  \GLo B).
\]
Then $\delta _n $ is given by
\[
\delta _n(a) =  - \alpha _\kappa a^{-1} \equiv - [\kappa, a]
\]
where    $ [\kappa, a]  $ is the \esp .
\end{enumerate}

\end{Thm}

\begin{proof}
We roughly sketch Thomsen's proof of (1) and (2).   We are in the classical situation with a 
topological group, a closed subgroup, and a quotient group. Steenrod's Bundle Structure Theorem 
(\cite{Steenrod}, p. 30) shows that the existence of continuous local sections is necessary and 
sufficient for $\diamond$ to be a fibre bundle. Thomsen produces these sections. The base 
space $ \Uo   B$ is metric, hence paracompact, and it is a standard fact (cf. Spanier \cite{Spanier}, Theorem 7.14, 
page 96) that a fibre bundle over a paracompact space is a fibration. 
This establishes (1) and (2).

Our proof of (3) follows   Wockel \cite{Wockel}, Theorem 2.3 in spirit.
Represent $a$ by 
\[
a : [0,1] \times S^{n-1} \to  \Uo   B 
\]
 with $a$  trivial on $ \{0,1\} \times S^{n-1}$. Without loss of generality
we may also assume that it is trivial on $[0,1] \times \{x_o\}$.   Define maps $A^\pm $ as follows:
\[
A^+ : D^n \times H^+ \to  \Uo   B \qquad\qquad  A^+(d, tx) =  \alpha_{\kappa (x)} (a(t(d)))
\]
\[
A^- : D^n \times H^- \to  \Uo   B \qquad\qquad  A^-(d, y) = a(d) .
\]
If $t = 1 $ then  $A^+(d, x) =\alpha _{\kappa(x)}(a(d))$ and $A^-(d, y) = a(d)$ as desired, and so the maps patch together
to form (after taking adjoints) a map
\[
A : D^n \longrightarrow  \Uo   A_{\zeta \otimes B}
\]
with the property that
\[
ev (A(d)) = A^- (d, 0) = a(d).
\]
Now collapse all of  $H^- $ to the south pole, the basepoint of $S^k$. The result is another copy of $S^k $, of course and
by definition $\delta _n$ is given by
\[
\delta _n(a) = \big[ A^+ \big| _{\partial D^n \times D^k }\big]    \in \big[  \partial D^n  \times S^k ,  \Uo   B \big]_*
\]
Define
$
\tilde A : D^n \times D^k \to \GLo B
$
by
\[
\tilde A (d,x) = A^+ (d,x) a(d)^{-1}
\]

Then:
\begin{enumerate}
\item   $\tilde A = * $ on $\partial D^n \times \partial D^k $.
\item    $\tilde A = * $ on $D^n \vee D^k$.
\item $\tilde A = A$ on $\partial D^n \times D^k$ because $a$ is trivial there.
\item $\tilde A (d,x) = [\kappa , a] $ on $D^n \times \partial D^k$ since $t = 1$ there.
\end{enumerate}
Thus
\[
\delta _n(a) =
 \big[ A^+ \big| _{\partial D^n \times D^k }\big]
 = \big[ \tilde A  \big| _{\partial D^n \times D^k }\big]
 =  - \big[ \tilde A  \big| _{ D^n \times \partial D^k }\big]
\]
since  $ \big[ \tilde A  \big| _{\partial (D^n \times D^k) }\big] = 0$
\[
  = - [\kappa , a].
\]
To complete the proof we note that the entire spectral sequence is natural under localization.

\end{proof}

\begin{Rem} If the analog of Theorem \ref{T:Wockel} holds for Banach algebras then the main theorems of this paper would similarly 
generalize. The obstacle would seem to be to generalize Thomsen's Theorem 1.9 of \cite{Thomsen} to this context; we believe that 
this is possible. 
\end{Rem}

\section{Proofs of the Main Theorems}

We have already done the heavy lifting for Theorem \ref{T:maintheoremhomotopy}. Here is the conclusion:

\vglue .2in
{\emph{\bf{Proof of Theorem \ref{T:maintheoremhomotopy}. }}} 
\begin{proof}
By naturality we may restrict to the case $\PP = \ZZ$. 
    Theorem \ref{T:theoremsphere} gives us a long exact sequence 
\[
\dots \to
\pi _n( \Uo   A_{\zeta \otimes B})
\overset{r}\to
\pi _n( \Uo   B)
\overset{d_k}\longrightarrow
\pi _{n+k-1}( \Uo   B)
\overset{s}\to
\pi _{n-1}( \Uo   A_{\zeta \otimes B})
\to\dots
\]
and so it suffices to show that $d_k(1\otimes a) = -g[\kappa, a]$.  Proposition \ref{T:differential} identifies the differential $d_k$ as the composite
 
 \[
 E^k_{0,q}
  \overset{\delta ^k}\longrightarrow
  D^k_{-1, q  }
\overset{\cong}\longleftarrow  \dots \overset{\cong}\longleftarrow
 D^k_{-k, q+k - 1 }
 \overset{ j^k  }\longrightarrow
  E^k_{-k, q+k - 1 }
\]
where
  \[
 E^2_{0,q} =  E^k_{0,q}   \cong H^0(S^k ; \ZZ )\otimes  \pi _q(\GLo B)
 \]
 \[
 E^2_{-k,q+k - 1} =   E^k_{-k,q+k - 1}   \cong H^k(S^k ; \ZZ )\otimes  \pi _{q+k - 1}( \Uo   B)  .
 \]
 Since $D^k_{-k, q+k - 1 } = D^1_{-k, q+k - 1 }$ in all intermediate parts of the filtration we can see that the differential is really induced by 
 the connecting homomorphism $\delta _n$ in the fibration associated with the fibration 
 \[
   \Uo  _\bullet A_{\zeta \otimes B}      \longrightarrow   \Uo   A_{\zeta \otimes B} \overset{r}\longrightarrow  \Uo   B   .
 \]
Theorem \ref{T:Wockel}  identifies this connecting homomorphism as an enhanced Samelson product  $\delta _n (a) = -[\kappa, a]$ 
and from there the proof is immediate.
 
\end{proof}

In order to prove Theorem \ref{T:maintheoremk}  we need some more information about  the special case $k = 3$ and we turn our attention to that case now.

Let  $k = 3$,  so that we focus on principal bundles over $S^3$.  Let  $B = \KK $, the compact operators on the standard Hilbert space $\HH $
 In this case the (contractible) unitary group $\UU = \UU (\HH )$ acts on $\KK $ by conjugation, its center $S^1 $ acts trivially, of course, and
 hence the action descends to an action of the projective  unitary group $\PU $ on $\KK $. Note that $\PU \simeq BS^1 \simeq K(\ZZ, 2)$.
  
 Let
 \[
   \U   \KK = \{ u \in \UU (\HH)  : u - 1 \in \KK \},
 \]
  the
 unitary group of $\KK $.  Recall that $\Omega ^2    \U   \KK \simeq   \U   \KK $ is one way
 of stating  Bott periodicity.

  Suppose that  $\zeta : P \to S^3$
is a principal $\PU$-bundle. This bundle is classified by a  clutching map $\kappa : S^2 \to \PU $ as per the notation above, and the homotopy class of $\kappa $
  in
  $  [S^2, \PU ] \cong H^2(S^2; \ZZ )$ has the form $ \Delta _\zeta g$  (where $g$ is a generic term for canonical generator) and so determines an integer $\Delta _\zeta $ which is essentially
  the Dixmier-Douady class of the principal bundle.

Let $A_{\zeta\otimes\KK}  $ denote the associated $C^*$-algebra.
 Then a consequence of the Theorem is that $\pi _*( \U   A_{\zeta\otimes\KK} )$ is determined up to group extension by the differential
 \[
\delta_3 : \pi _n(  \U   \KK) \longrightarrow \pi _{n+2}(  \U   \KK).
\]
These groups are zero for $n$ even and $\ZZ $ for $n $ odd by Bott periodicity.  If we see this in terms of $E_2$ it corresponds to
\[
d _3 :  H^0(S^3) \otimes \pi _n(  \U   \KK) \longrightarrow H^3(S^3) \otimes \pi _{n+2}(  \U   \KK)
\]
and we have proved that
\[
d_3 (1 \otimes a) = -g [ \kappa , a ].
\]
 
\begin{Pro} With the notation above, for $n$ odd,
\[
[\kappa , a ] = \Delta _\zeta  \hat{\beta}(a)
\]
where $ \hat{\beta}(a)$ is the composite
\[
S^{n+2} \simeq S^2 \wedge S^n \overset{1 \wedge a}\longrightarrow  S^2 \wedge   \U   \KK  \overset{\hat{\beta}}\longrightarrow    \U   \KK
\]
and $\hat{\beta} : S^2 \wedge   \U   \KK  \to   \U   \KK$ is the adjoint of the Bott periodicity identification $\beta :    \U   \KK \simeq \Omega ^2  \U   \KK $.
\end{Pro}

\begin{proof}
Regard $\kappa : S^2 \to \PU    \simeq BS^1 $ as a line bundle $L \to S^2 $ with first Chern class $c_1(L) = \Delta _\zeta g$.  Regard $a: S^{n} \to   \U   \KK $ as the clutching
map of a $\KK$- bundle $F \to S^{n+1} $. Then the bundle $L\otimes F \to S^2\wedge S^{n+1} \cong S^{n+3}$ is represented by the clutching map
\[
\kappa (x)a(y)\kappa(x)^{-1} : S^2 \wedge S^{n} \longrightarrow   \U   \KK .
 \]
Then we appeal to the argument of Proposition 2.1 of Atiyah-Segal \cite{AS}. They note that $L\otimes F $ is a sub-bundle of the trivial bundle $\HH \otimes F$ and
hence
\[
\kappa (x)a(y)\kappa(x)^{-1}a(y)^{-1} = \Delta _\zeta (1 \wedge a(y)^{-1})  .
\]
  Using notation introduced previously, we rewrite this as
 \[
[\kappa , a ] = \Delta _\zeta    \hat{\beta}(a)
\]
and this proves the proposition.

\end{proof}

{\emph{\bf{Proof of Theorem \ref{T:maintheoremk}. }}}

 \begin{proof}  There is a natural isomorphism (or definition, depending upon how you set up $K_*$ )
 \[
 K_{n+1}(A)   \cong \pi _{n}( \U   (A \otimes  \KK )) \cong \Dirlim _j \pi _{n}( \U   (A \otimes M_j   )).  \qquad (n \geq 0) 
 \]           
Fix some topological group $G$, a principal $G$-bundle $P \to S^k$ and let $G$ act on a Banach algebra $B$ as usual. Then $G$ acts on $B\otimes M_j$ by acting 
trivially on the second factor and the (non-unital) inclusions $M_j \to M_{j+1}$ that insert the matrix ring into the top left of the next matrix ring induce $G$-equivariant maps 
$B\otimes M_j \to B\otimes M_{j+1}$.  For brevity in this proof we write $B_j = B\otimes M_j$. This induces a sequence of morphisms
\[
A_{\zeta\otimes B} \longrightarrow A_{\zeta\otimes B_2}\longrightarrow A_{\zeta\otimes B_3}\longrightarrow \dots A_{\zeta\otimes B_j}\longrightarrow \dots
\]
Each of these algebras has an associated Wang sequence   (\ref{T:maintheoremhomotopy}) and the naturality of the spectral sequence that gave them birth yields a commuting diagram
\vglue .1in
\[
\begin{CD}
  @>>>    \pi _n(   \Uo   A_{\zeta \otimes B})  @>r>>   \pi _n(  \Uo   B)  @>d_k>> \pi _{n+k-1}(\  \Uo   B) @>s>>  \pi _{n-1}(\  \Uo   A_{\zeta \otimes B}) @>>>  \\
 @.             @VVV       @VVV          @VVV          @VVV     \\
   @>>>    \pi _n( \Uo   A_{\zeta \otimes B_2})  @>r>>   \pi _n(  \Uo   B_2)  @>d_k>> \pi _{n+k-1}(\  \Uo   B_2) @>s>>  \pi _{n-1}( \Uo   A_{\zeta \otimes B_2}) @>>>  \\
  @.             @VVV       @VVV          @VVV          @VVV     \\
   @>>>    \pi _n( \Uo   A_{\zeta \otimes B_3})  @>r>>   \pi _n(  \Uo   B_3)  @>d_k>> \pi _{n+k-1}(\  \Uo   B_3) @>s>>  \pi _{n-1}( \Uo   A_{\zeta \otimes B_3}) @>>>  \\
  @.             @VVV       @VVV          @VVV          @VVV     \\
   @>>>    \pi _n( \Uo   A_{\zeta \otimes B_j})  @>r>>   \pi _n( \Uo   B_j)  @>d_k>> \pi _{n+k-1}( \Uo   B_j) @>s>>  \pi _{n-1}( \Uo   A_{\zeta \otimes B_j}) @>>>  \\
  @.             @VVV       @VVV          @VVV          @VVV     \\
   \end{CD}
\]
\vglue .1in
Taking direct limits over $j$ preserves exactness, and thus there is a long exact sequence
\[
\dots \to
K_n(A_{\zeta \otimes B})
\overset{\rho}\to
K_n(B)  
\overset{d_k}\longrightarrow
K _{n+k-1}(B)
\overset{\sigma}\to
K _{n-1}( A_{\zeta \otimes B})
\to\dots
\]
If $B$ is a $C^*$-algebra then the differential $d_k$ in the homotopy sequences is given by
\[
 d_k(1 \otimes a)  = -\,g\otimes [ \kappa, a ]
 \]
 where $g$ is the generator of $H^k(S^k;\ZZ )$ and $[ \kappa, a ]$ is the \esp 
 \, with the class $\kappa : S^{k-1} \to G$ that classifies the principal bundle and this passes to direct limits as well.  Similarly we may apply $(-)\otimes\PP$
  to obtain the full result with coefficients.

  The only remaining issue is the explicit identification of the differential in the case
  $k = 3$ and that is done in the previous proposition.
 \end{proof}

\section{An Example}

In this section we illustrate our result in a very concrete case. Take the principal bundle to be the Hopf bundle
\[
\zeta : \,\,\,\, S^7 \longrightarrow S^4
\]
obtained from the multiplicative structure of the quaternions,
with group $G = S^3 = S\U_2$.  Take $B = M_2$ 
(the $2\times 2$ complex matrices)with
$S^3$ acting upon  $M_2$    by conjugation. Then there is
an associated bundle of $C^*$-algebras
\[
{\zeta \otimes M_2}  :   \qquad   S^7 \times _{S^3} M_2 \longrightarrow S^4
\]
and as usual we denote by $A_{\zeta \otimes M_2} $ the associated $C^*$-algebra of continuous sections.  We want to compute $\pi _*( \U   A_{\zeta \otimes M_2} )$. The Wang sequence then takes the form
\[
\longrightarrow \pi _n( \Uo A_{\zeta \otimes M_2} ) \longrightarrow \pi_n(\U_2) \overset{d_4}\longrightarrow \pi _{n+3}(\U_2) \longrightarrow  \pi _{n-1}(  \U   A_{\zeta \otimes M_2} )\longrightarrow
\]
where $ \U _n =  \U   (M_n)$.
 Recall that $\U_2 \cong S^1 \times S^3 $ as topological
spaces, though not as groups. Serre's classical results on homotopy imply that $\pi _n(\U_2)$ is a  finite group  for each $n > 3$
 and that these are groups are non-zero for infinitely many values of $n$.

We record for reference the first twelve homotopy groups of $\U_2$. Let $\eta \in \pi _3(S^2)$ denote the Hopf generator and
by abuse of notation  its various suspensions, so for instance we write $\eta ^2 \in \pi _5(S^3)$ for the composition
\[
S^5 \overset{S^2\eta}\longrightarrow S^4 \overset{S\eta}\longrightarrow S^3 .
\]

 We use $a_n$ as labels for classes when there don't seem to be
standard names; the subscript denotes the homotopy group.

\begin{itemize}
\vglue .1in\item  $\pi _1( \U_2 ) \cong \pi _1(S^1) \cong \ZZ  $ on  the class of the upper left corner inclusion
    $
    S^1 \to \U_2
$ which we denote $a_1$.
In all higher degrees the natural inclusion  $S^3 \cong S\U_2 \to \U_2$ induces an isomorphism in homotopy.
\vglue .1in\item  $\pi _2(\U_2 ) =  0$.
\vglue .1in\item $\pi _3( \U_2 ) \cong  \ZZ $.  The generator  is given by the natural inclusion.  
\[
\iota : S^3 \cong S\U_2 \to \U_2 .
 \]
\vglue .1in\item $\pi _4( \U_2 ) \cong  \ZZ /2$ on the class $\eta $.
\vglue .1in\item $\pi _5( \U_2 ) \cong  \ZZ /2$ on the class $\eta ^2$.
\vglue .1in\item $\pi _6( \U_2 ) \cong  \ZZ /12$ on the class $a_6$. The $2$-primary part is generated by a class  $\nu '$ (in Toda's \cite{Toda} notation).
\vglue .1in\item $\pi _7( \U_2 ) \cong  \ZZ /2$ on the class $\nu '\eta $.
\vglue .1in\item $\pi _8( \U_2 ) \cong  \ZZ /2$ on the class $\nu ' \eta ^2$.
\vglue .1in\item $\pi _9( \U_2 ) \cong  \ZZ /3$ on the class $a_9$. (J. C. Moore \cite{Moore}, Theorem 5.3.)
\vglue .1in\item $\pi _{10}( \U_2 ) \cong  \ZZ /15$  on the class  $a_{10}$.  (J. C. Moore \cite{Moore}, Theorem 5.3 and Lemma 5.1.)
\vglue .1in\item $\pi _{11}( \U_2 ) \cong  \ZZ /2$ (Toda \cite{Toda} Theorem 7.2).
\vglue .1in\item $\pi _{12}( \U_2 ) \cong  (\ZZ /2)^2$   (Toda \cite{Toda} Theorem 7.2).

\end{itemize}

Recall that we have shown that the differential $d_4$ is given by
\[
d_4 (a) = -g [\kappa , a]
\]
where $\kappa $ is the clutching map of the principal bundle. In this example, we have:

\begin{Pro} The clutching map of the Hopf bundle $S^7 \to S^4 $ is the identity map $\iota :  S^3 \to S^3 $.
\end{Pro}

We are indebted to John Klein for the following proof of this fact.
\begin{proof} This fact is a general characteristic of the Hopf construction, in the case $G = S^3$ with its standard multiplication.
If $G$  is a topological group with  multiplication $G \times  G \longrightarrow G$, one has a Hopf construction
\[
G \star G \longrightarrow SG
\]
where $\star $ denotes join and $SG$ is the unreduced suspension of $G$.
This is a fibre bundle with fibre at the basepoint $G$. The bundle projection is given by
 \[
  tg + (1-t)h  \longrightarrow   t(gh) \qquad    t \in [0,1],\,\, g,h \in  G.
\]
The clutching map in this case is given by the map $G  \longrightarrow  homeo(G)$ which is adjoint to left
multiplication.
This factors through the identity map of $G$ considered as acting by left multiplication on itself which
shows that the fibration has a reduction of structure group to $G$ and has clutching map $\iota  : G \to G$.
\end{proof}

 In light of the Proposition, we see that in our example the differential is given by
 \[
d_4 (a) = -g [\iota , a]
\]
where $\iota : S^3 \to S^3$ is the identity map. So we must calculate the Samelson product
\[
[\iota , - ] : \pi _n(\U_2) \longrightarrow \pi _{n+3} (\U_2).
\]
Here is the result. Note that each entry that is non-zero corresponds to a non-zero $d_4$ differential.
\begin{itemize}
\vglue .1in\item $n=1:$\,\, $    [\iota , a_1 ] =  0$
\vglue .1in\item  $n=3:$   \,\,           $ [\iota , \iota ] = a_6 $ by the result of I. M. James \cite{james}, p. 176.
\vglue .1in\item  $n=4:$\,\,   $[\iota , \eta ] = \nu '\eta$ since  (working $2$-primary)
\[
[\iota , \eta  ] = [\iota , \iota] \circ \eta  =  \nu '\eta .
\]

\vglue .1in\item $n=5:$\,\,   $[\iota , \eta ^2 ] =  \nu ' \eta ^2$ by the same argument.
\vglue .1in\item $n=6:$\,\,   $[\iota , a_6] =   [\iota , [\iota , \iota ] ] = a_9 $ by I. M. James \cite{james2}, \S 3.
\vglue .1in\item $n=7:$\,\,   $[\iota ,\nu '\eta ] = [\iota , \nu ' ]\circ \eta  = a_9 \circ \eta  = 0$ (  since $a_9$ has order $3$ and $\eta $ has order $2$.)
\vglue .1in\item $n=8:$\,\,   $[\iota ,\nu '\eta ^2] = 0$ by same argument.
\vglue .1in\item $n=9:$\,\,  $[\iota , a_9 ]  = 0 $  by nilpotency.

 \end{itemize}
 Feeding this information into the Wang long exact sequence produces the following results:

 \begin{Thm}  Let
\[
\zeta : \,\,\,\, S^7 \longrightarrow S^4
\]
denote the Hopf bundle. Form the associated bundle of $C^*$-algebras
\[
S^7 \times _{S^3} M_2 \longrightarrow S^4
\]
and let $A_{\zeta \otimes M_2} $ denote the associated $C^*$-algebra of continuous sections. Then  $\pi _*( \U   _\bullet A_{\zeta \otimes M_2} )$ is given  as follows:
\begin{itemize}
\vglue .1in\item $\pi _1( \Uo  A_{\zeta \otimes M_2} )$ fits into a split short exact sequence
\[
0 \to \pi _5(\U_2) \longrightarrow \pi _1( \Uo  A_{\zeta \otimes M_2} ) \longrightarrow \pi _1(\U_2) \to 0
\]
with $\pi _5(\U_2) \cong \ZZ /2 $ and $\pi _1(\U_2) \cong \ZZ $  and so
$\pi _1( \Uo  A_{\zeta \otimes M_2} )   \cong \ZZ \oplus \ZZ /2 .$

\vglue .1in\item $\pi _2( \Uo  A_{\zeta \otimes M_2} ) = 0$.

\vglue .1in\item $\pi _3( \Uo  A_{\zeta \otimes M_2} ) $ fits in a short exact sequence
\[
0  \to  \pi _3( \Uo  A_{\zeta \otimes M_2} )  \longrightarrow \pi _3(\U_2) \longrightarrow \pi _6(\U_2) \to 0
\]
with
$\pi _3(\U_2) \cong \ZZ $
and
$\pi _6(\U_2) \cong \ZZ /12$,
and so $\pi _3( \Uo  A_{\zeta \otimes M_2} )  \cong \ZZ $.

\vglue .1in\item $\pi _4( \Uo  A_{\zeta \otimes M_2} ) = 0$.

\vglue .1in\item $\pi _5( \Uo  A_{\zeta \otimes M_2} ) = 0 $.

\vglue .1in\item $\pi _6( \Uo  A_{\zeta \otimes M_2} )$ fits in a short exact sequence
\[
0 \to \pi _{10}(\U_2) \longrightarrow \pi _6( \Uo  A_{\zeta \otimes M_2} ) \longrightarrow \ZZ/4  \to 0
\]
with
$\pi _{10}(\U_2) \cong \ZZ/15 $ and $\ZZ/4 $ the $2$-primary component of $\pi _6(\U_2)$, so that
$\pi _6( \Uo  A_{\zeta \otimes M_2} ) \cong \ZZ/60$.
\vglue .1in\item $\pi _7( \Uo  A_{\zeta \otimes M_2} )$ fits in a short exact sequence
\[
0 \to \pi _{11}(\U_2) \longrightarrow \pi _7( \Uo  A_{\zeta \otimes M_2} ) \longrightarrow \pi _{7}(\U_2)  \to 0
\]
with
$\pi _{11}(\U_2) \cong \ZZ/2 $ and   $\pi _{7}(\U_2) \cong \ZZ/2 $.  So $\pi _7( \Uo  A_{\zeta \otimes M_2} )$ is  either $\ZZ/4 $ or
$(\ZZ /2)^2$.
\vglue .1in\item $\pi _8( \Uo  A_{\zeta \otimes M_2} )$ fits in a short exact sequence
\[
0 \to \pi _{12}(\U_2) \longrightarrow \pi _8( \Uo  A_{\zeta \otimes M_2} ) \longrightarrow \pi _{8}(\U_2) \to 0
\]
with
$\pi _{12}(\U_2) \cong (\ZZ/2)^2 $ and   $\pi _{8}(\U_2) \cong \ZZ/2 $.  So $\pi _8( \Uo  A_{\zeta \otimes M_2} )$ is  either $\ZZ/4 \oplus \ZZ /2$ or
$(\ZZ /2)^3$.
 \end{itemize}
\end{Thm}

We contrast this with the analogous computation in rational homotopy and in $K$-theory.

\begin{Thm}
With the notation above,
\begin{enumerate}
\item The rational homotopy groups of $ \Uo  A_{\zeta \otimes M_2} $ are zero except for
\[
\pi _j( \Uo  A_{\zeta \otimes M_2} )\otimes\QQ \cong \QQ \qquad\qquad j = 1\,\text{and}\,\,3.
\]
\item The (matrix) stable homotopy groups are zero in even degrees and
\[
\pi _j( \Uo  (A_{\zeta \otimes M_2} \otimes\KK   ))  \cong \ZZ \qquad\qquad j \,\, \text{odd}
\]
\item The $K$-theory groups are given by $K_0(A_{\zeta \otimes M_2} ) \cong \ZZ $ and $K_1(A_{\zeta \otimes M_2} ) = 0$.
\end{enumerate}
\end{Thm}

\begin{proof}
In the rational homotopy case the only non-zero homotopy groups of $\U_2$ are
\[
\pi _1(\U_2)\otimes\QQ  \cong \pi _3(\U_2)\otimes\QQ  \cong \QQ .
\]
Since the Wang differential changes degree by three, 
the differential 
 must be identically zero, the spectral sequence collapses, and the long exact
Wang sequence turns into many short exact sequences. As $\pi_j(\U_2) = 0$ except for $j = 1, \,\,3$, the result is as stated.

In the stable case the situation is similar, since stably $\pi _*( \Uo  (M_2\otimes \KK)) \cong \ZZ $ for $*$ odd and zero for $*$ even by Bott periodicity.
Again, the Wang differential is identically zero and the result follows. Part (3) follows from (2) essentially by definition.
\end{proof}

The following table summarizes the calculations in this section. 
 \vglue .3in

\begin{centerline}{\bf{The homotopy, rational homotopy and K-theory}}
\end{centerline}
\begin{centerline}{\bf{of $A_{\zeta\otimes M_2}$ in low degrees}}
\end{centerline}
\vglue .2in
\begin{center}
\begin{tabular}{  |l  | c | c  | c | }
\hline
n    &   $ \pi _n( \Uo  A_{\zeta\otimes M_2 }) $    &  
 $ \pi_n( \Uo  A_{\zeta\otimes M_2}) \otimes\QQ  $  &  $ K_{n+1}(A_{\zeta\otimes M_2})$   \\   \hline\hline
  &  &  &  \\ \hline
1    &   $\ZZ \oplus \ZZ/2$   &  $\QQ$   &  $\ZZ$   \\   \hline
2    &   0   &   0 &  0   \\   \hline
3    &   $\ZZ$   &   $\QQ $  &  $\ZZ$     \\   \hline
4    &   0   &   0 &  0   \\   \hline
 5   &   0   &   0 &  $\ZZ $   \\   \hline
 6    &   $\ZZ /60$    &   0 &  0   \\   \hline
 7    &   $\ZZ /4 $  or  $(\ZZ/2)^2 $  &   0 &  $\ZZ $  \\   \hline
 8    &   $\ZZ/4 \oplus \ZZ/2$ or $(\ZZ/2)^3$     &   0 &  0   \\   \hline

\end{tabular}
\end{center}
\vglue .2in
We note several features of the table:
\begin{enumerate}
\item The right column is periodic, reflecting Bott periodicity.
\item The center column shows the beginning of periodicity - entries in degrees 1 and 3 will persist to the direct limit. However in degree 5 the rational 
homotopy vanishes, reflecting the fact that the matrix ring is not large enough to pick up the classes that will eventually generate the stable periodic elements.
\item The calculations suggest that perhaps $\pi _n( \U   _\bullet A_{\zeta\otimes M_2} \otimes\QQ ) = 0$ for $n > 3$.  In fact this is true, by \cite{KSS}.  
This implies that $ \pi _n( \U   _\bullet A_{\zeta\otimes M_2}) $ is finite for each $n > 3$.  
\item As $\pi _n(\U_2) \neq 0$ for infinitely many values of $n$, we would
suppose that the same is true for $ \pi _n( \U   _\bullet A_{\zeta\otimes M_2}) $.
 
\end{enumerate}

We regard this example as an excellent illustration of what is lost by focusing attention only upon $K_*(A_{\zeta \otimes B} )$.
The richness of detail that is evident while studying the individual homotopy groups is completely lost upon matrix stabilization and
passage to $K$-theory.




\end{document}